\documentclass[reqno, 11pt]{amsart}

\usepackage{setspace}
\usepackage{comment}
\usepackage{biblatex}
\usepackage{amsmath} 
\usepackage{amsthm}
\usepackage{amssymb}
\usepackage{hyperref}
\usepackage{color}
\usepackage{tikz}
\usepackage{setspace}
\usepackage{graphicx}
\usepackage{subcaption}
\usepackage{caption}
\usepackage{enumitem}
\usepackage{changepage}

\calclayout

\DeclareFieldFormat{doi}{\href{https://dx.doi.org/#1}{\textsc{doi}}}
\DeclareFieldFormat{url}{\href{#1}{Link}}

\newcommand{\GPS}{GPS }

\newtheorem{theorem}{Theorem}[section]

\newtheorem{lemma}[theorem]{Lemma}

\newtheorem{cor}[theorem]{Corollary}

\theoremstyle{definition}
\newtheorem{definition}[theorem]{Definition}
\newtheorem{notation}[theorem]{Notation}
\newtheorem{remark}[theorem]{Remark}
\newtheorem{example}[theorem]{Example}

\newcommand{\al}{\alpha}
\newcommand{\be}{\beta}

\newcommand{\eps}{\varepsilon}

\newcommand{\mA}{\mathcal{A}}

\newcommand{\mC}{\mathcal{C}}

\newcommand{\bm}{\textbf{m}}

\newcommand{\R}{\mathbb{R}}

\newcommand{\td}{;}
\newcommand{\bS}{\textbf{S}}

\newcommand{\Ds}{D_{\textbf{S}}}
\newcommand{\Ms}{\mathcal{M}_{\bS}}
\newcommand{\Ls}{\mathcal{L}_{\bS}}

\newcommand{\sse}{\subseteq}
\newcommand{\sli}{\sum\limits}

\newcommand{\parkS}{\text{Park}_{\bS}}
\newcommand{\As}{\mathcal{A}_{\bS}}
\newcommand{\trips}{\text{Triple}_{\bS}}
\newcommand{\Ns}{\mathcal{N}_{\bS}}
\newcommand{\me}{(\bm, \eps)}
\newcommand{\ra}{\rightarrow}

\newcommand{\ZZ}{\mathbb{Z}}

\title[]{Bijectivity of a generalized Pak-Stanley labeling}
\author{Olivier Bernardi}
\author{Neha Goregaokar}
\date{\today}

\address{Department of Mathematics, Brandeis University, Waltham, MA 02453, USA}
\email{bernardi@brandeis.edu}
\email{ngoregaokar@brandeis.edu}

\begin{document}

\begin{abstract}
The Pak-Stanley labeling is a bijection between the regions of the $m$-Shi arrangement and the $m$-parking functions. Mazin generalized this labeling to every deformation of the braid arrangement and proved that this labeling is always surjective onto a set of directed multigraph parking functions. We provide a right inverse to the generalized Pak-Stanley labeling, and identify a class $\mC$ of arrangements for which this labeling is bijective. The class $\mC$ includes the multi-Shi arrangements and the multi-Catalan arrangements. We also show that the arrangements in $\mC$ are the only transitive arrangements for which the generalized Pak-Stanley labeling is bijective.
\end{abstract}

\maketitle

\section{Introduction}

A (real) hyperplane arrangement is a finite collection of affine hyperplanes in $\R^n$ for some integer $n \geq 1$. The hyperplanes decompose the space into convex \textit{regions}. It is a central question in the study of hyperplane arrangements to count the number of regions of a given hyperplane arrangement, and to provide a bijective encoding of these regions by some combinatorial objects. 

For the Shi arrangement, Pak suggested a labeling of the regions using parking functions~\cite{PS1}. Stanley extended this labeling to $m$-Shi arrangements~\cite{Stanley1998}, and proved that it is a bijection between the regions of the $n$ dimensional $m$-Shi arrangement and the $m$-parking functions of length $n$. 

Subsequently, there were several attempts to generalize the Pak-Stanley bijection to other arrangements. Duval, Klivans, and Martin~\cite{DKM_Conjecture} conjectured that for every graph $G$ the Pak-Stanley labeling was a surjection between the regions of the $G$-Shi arrangement and the $G$-parking functions. This was proved by Hopkins and Perkinson who treated the case of general bigraphical arrangements~\cite{HP_Bigraphical,HP_Semiorder}. 

In~\cite{MR3721647}, Mazin generalized the Pak-Stanley labeling to every \emph{deformation of the braid arrangement}. These are the arrangements in which every hyperplane has the form   
$$H_{i,j,s} := \{(x_1, \ldots ,x_n) \in \R^n \mid x_i - x_j = s\},$$ 
for some indices $i,j\in[n]$ and real number $s$. To any such arrangement $\mA$ one can associate a directed multigraph $D_{\mA}$, and Mazin showed that the generalized Pak-Stanley labeling is always surjective onto the set of $D_{\mA}$-parking functions~\cite{MR3721647} (see Section \ref{sec:Mazin} for the definition). 

Importantly, the generalized Pak-Stanley labeling is not bijective in general, and it is an open question to characterize the arrangements for which it is bijective. In~\cite{Baker}, Baker proved a necessary (but not sufficient) condition for injectivity (see also the slight improvement in~\cite[Theorem 1.30]{Miller}). Mazin and Miller then showed that for central graphical arrangements, this condition is sufficient in~\cite{MMGraphical}.

In this paper, we define a right inverse to the generalized Pak-Stanley labeling. This gives a new proof of Mazin's surjectivity result~\cite{MR3721647}.
We then show that the generalized Pak-Stanley labeling is bijective for a family of arrangements that generalize the $m$-Shi and $m$-Catalan arrangements, and that we call $\me$-arrangements. Let us introduce some notation to describe this family.

\begin{notation}
Let $\bS = (S_{i,j})_{1\leq i < j \leq n}$ be a collection of finite sets of integers. We define the \emph{$\bS$-braid arrangement} $\mA_{\bS}$ as the collection of the following hyperplanes:
    $$\mA_{\bS} = \{H_{i,j,s} \mid 1 \leq i < j \leq n, \, s \in S_{i,j}\}.$$
For $i < j $, we denote $S_{i,j}^+ = \{s > 0 \mid H_{i,j,s} \in \mA_{\bS}\}$ and $S_{j,i}^+ = \{s \geq 0 \mid H_{j,i,s} \in \mA_{\bS}\}$. 
\end{notation}
The \emph{$m$-Shi arrangement} is the arrangement $\mA_\bS$ for $S_{i,j}=[-m+1;m]$ for all $i<j$. The \emph{$m$-Catalan arrangement} is  the arrangement $\mA_\bS$ for $S_{i,j}=[-m;m]$ for all $i<j$.

\begin{remark}\label{rk:S-is-general}
Any deformation of the braid arrangement is combinatorially equivalent to an arrangement where the hyperplanes are of the form $H_{i,j,s}$ with $s\in \ZZ$, and any such arrangement is of the form $\mA_\bS$ for some tuple $\bS$ of sets of integers. 
\end{remark}

\begin{definition}\label{mepsarr}
A \emph{$\me$-arrangement} is an $\bS$-braid arrangement for a tuple $\bS = (S_{i,j})_{1\leq i < j \leq n}$ such that 
 $$S_{i,j}^+ = \begin{cases}
         [1; m_j - \eps_{i,j}] & \text{ if } i<j \\
         [0; m_j - \eps_{i,j}] & \text{ if } i>j 
     \end{cases},$$
where $\bm = (m_1, \ldots, m_n)$ is a tuple of non-negative integers, and  $\eps=(\eps_{i,j})_{i\neq j\in[n]}$ is a tuple of numbers in $\{0,1\}$ such that for all $i,j,k\in [n]$ with $i<j$ one has $\eps_{i,k}\leq \eps_{j,k}$.  \end{definition}

\begin{figure}[h]
    \centering
    \includegraphics[width=0.9\linewidth]{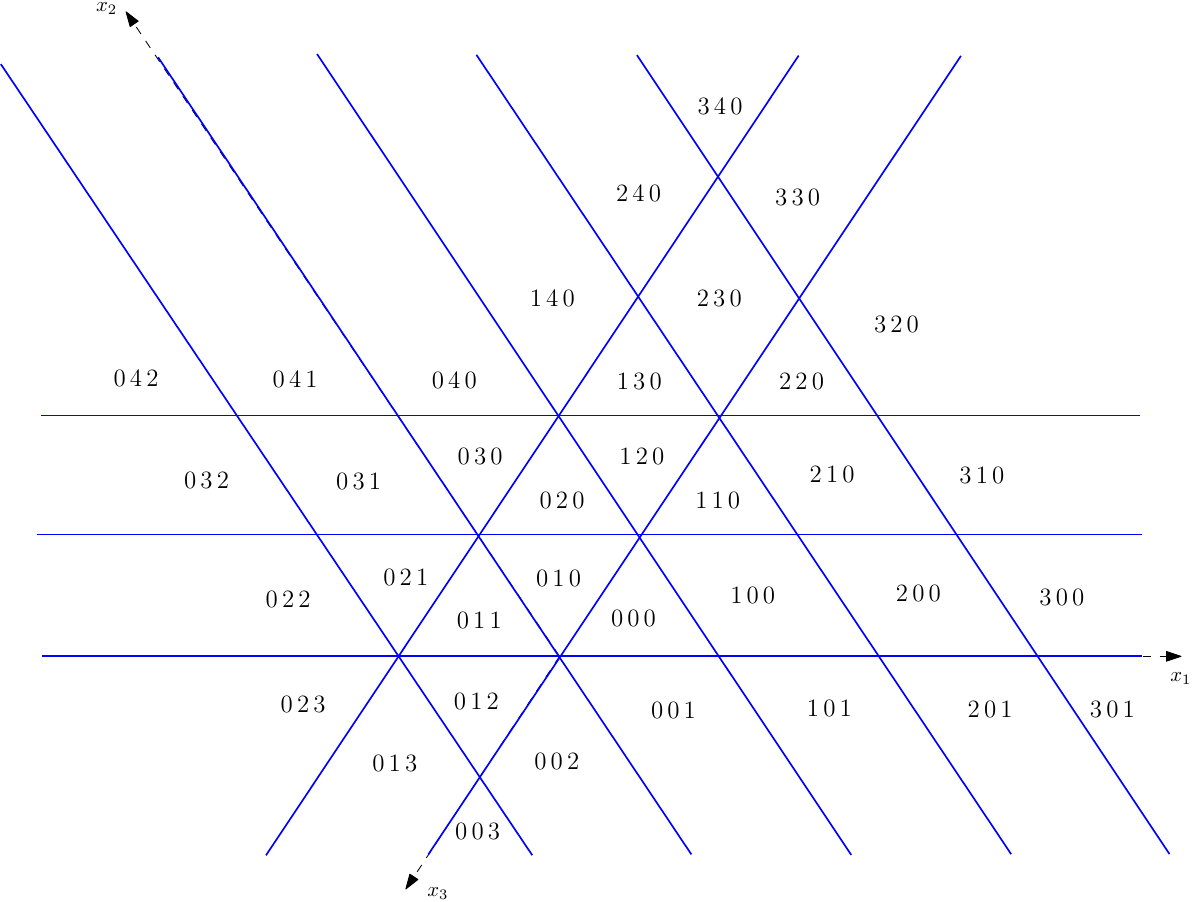}
    \caption{An $\me$-arrangement with its generalized Pak-Stanley labeling. Here we have $\bm = (1,0,3)$, $\eps_{2,3} = 1$, and $\eps_{i,j} = 0$ for all $(i,j)\neq (2,3)$.}
    \label{fig:mearr}
\end{figure}
An example of an $\me$-arrangement with its generalized Pak-Stanley labeling is shown in Figure~\ref{fig:mearr}. Note that if $m_j = m$ for all $j$ and $\eps_{i,j} = 0$ for all $i,j$, then the $\me$-arrangement is the $m$-Catalan arrangement. 
If $m_j = m$ for all $j$, and $\eps_{i,j} = 0$ if and only if $i<j$, then the $\me$-arrangement is the $m$-Shi arrangement. Let us also mention that the $\me$-arrangements appear in \cite[Section 3]{BeDo:reflection-arrangments}, where a bijective method is devised to count their regions.

In this paper, we establish the bijectivity of the generalized Pak-Stanley labeling for every $\me$-arrangement. In particular we obtain a new proof of the bijectivity of the Pak-Stanley labeling for $m$-Shi arrangements. 

From our bijectivity result, one can express the number of regions in any $\me$-arrangement as a number of multigraph parking functions, which admits a determinental formula via the matrix-tree theorem (see~\cite{MR2052943}).
The $\me$-arrangements satisfy a condition called \emph{transitivity}~\cite{Bernardi} (see Section \ref{sec:ber}) and we show that they are the only transitive arrangements for which the generalized Pak-Stanley labeling is bijective.

This paper is structured as follows. In Section~\ref{bgsec}, we set our notation and recall some known results. In Section~\ref{RightInverseSec}, we define a map that is a right inverse for the generalized Pak-Stanley labeling, thus recovering Mazin's surjectivity result. In Section~\ref{suffconds} we establish the bijectivity of the labeling for $\me$-arrangements. In Section~\ref{TransitiveArrSec}, we show that $\me$-arrangements are the only transitive arrangements such that the labeling is bijective, and recover Mazin and Miller's result for transitive central graphical arrangements.

\section{Background and notation}\label{bgsec}

In this section we state some definitions and recall some known results. For the rest of this paper, for a positive integer $n$, we denote $[n] := \{1, \ldots, n\}$ and for integers $a$ and $b$, we denote $[a;b]:= \{a, \ldots, b\}$. 
Further, $\bS = (S_{i,j})_{1\leq i < j \leq n}$ denotes a collection of finite sets of integers, and $R(\As)$ denotes the set of regions of the arrangement $\As$. 

\subsection{Parking functions and the generalized Pak-Stanley labeling}\label{sec:Mazin}

\begin{definition}
    Let $\bS = (S_{i,j})_{1 \leq i< j \leq n}$ be a collection of finite sets of integers. We define $\Ds$ to be the directed multigraph on $[n+1]$ with the arc set containing precisely $|S_{i,j}^+|$ arcs $(i,j)$ and one arc $(i, n+1)$ for all $i,j \in [n]$.  
\end{definition}

\begin{figure}[h]
    \centering
    \includegraphics[width=0.3\linewidth]{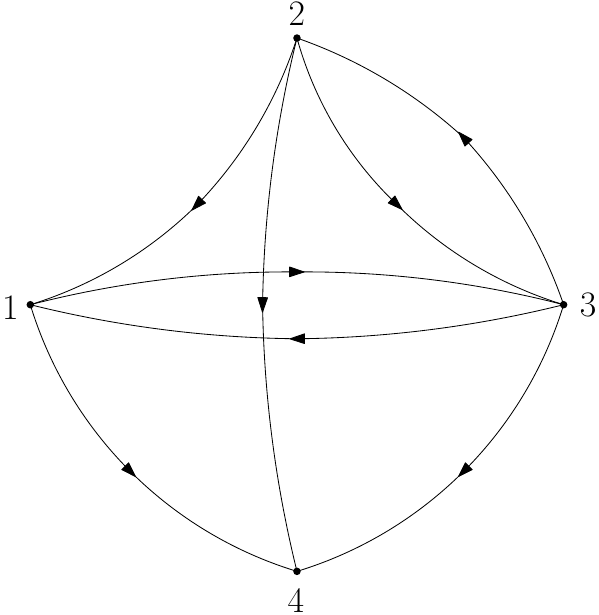}
    \caption{The directed multigraph $\Ds$ for $S_{1,2} = \{0\}$, $S_{1,3} = S_{2,3} = \{0,1\}$.}
    \label{fig:multigraph}
\end{figure}

\begin{definition}
    Let $D = ([n+1], A)$ be a directed multigraph. A sequence $p = (p_1, \ldots, p_n)$ is a \emph{$D$-parking function of length $n$} if and only if for every non-empty set $U \sse [n]$, $$\exists u \in U \text{ such that } p_u < |\{(u, v) \in A \mid v \notin U\}|.$$ 
\end{definition}

We denote by $\parkS$ the set of $\Ds$-parking functions. Then, the \emph{generalized Pak-Stanley labeling}, or \emph{\GPS labeling} for short, is the map $\lambda: R(\As) \ra \parkS$ defined as follows.

We denote by $R_0$ the region of $\mA_{\bS}$ containing the fundamental alcove $\{(x_1, \ldots, x_n) \in \R^n \mid x_1 > x_2 > \ldots > x_n > x_1 - 1\}$. Then, for a region $R$ of $\mA_{\bS}$, we define $\lambda(R) = (p_1, \ldots, p_n)$ with 
\begin{align*}
    p_i &= \# H_{i,j,s} \in \mA_{\bS} \text{ separating $R$ from $R_0$ with } s> 0 \text{ or } (s= 0 \text{ and } i>j) \\
    &= |\{j \in [n], s \in S_{i,j}^+ \mid H_{i,j,s} \text{ separates $R$ from $R_0$}\}|. 
\end{align*}

\begin{figure}[ht]
    \centering
    \includegraphics[width=\linewidth]{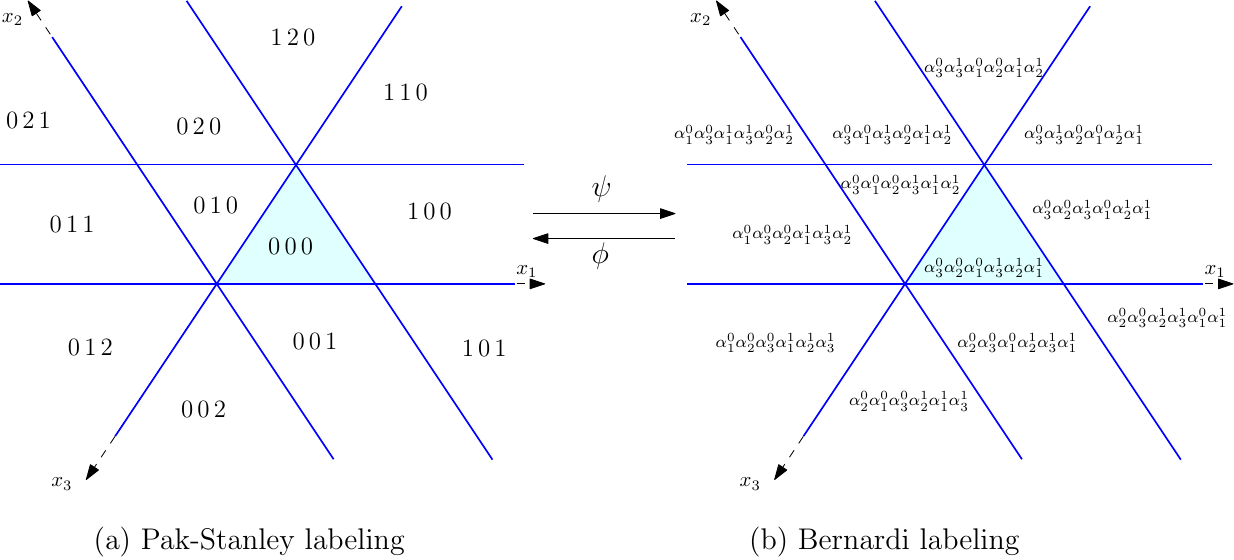}
    \caption{The generalized Pak-Stanley labeling, and Bernardi labeling for $\As$ with $S_{1,2} = \{0\}$, $S_{1,3} = S_{2,3} = \{0,1\}$. The region $R_0$ is shaded. The associated directed multigraph $\Ds$ is shown in Fig.~\ref{fig:multigraph}.}
    \label{fig:labelings}
\end{figure}

Mazin proved the following result in~\cite{MR3721647}.
\begin{theorem}[{~\cite[Theorem 4.4]{MR3721647}}]
For any tuple $\bS=(S_{i,j})$, the \GPS labeling $\lambda : R(\As) \ra \parkS$ is surjective. 
\end{theorem}

\subsection{Sketches and Bernardi labeling}\label{sec:ber}\hfill\\ 
In \cite{Bernardi}, Bernardi defined a bijection between regions of the $m$-Catalan arrangement and $(m,n)$-sketches. 

\begin{definition}
    An \emph{$(m,n)$-sketch} is a word $w = w_1w_2 \ldots w_{(m+1)n}$ such that 
    \begin{itemize}[nolistsep]
        \item $\{w_1, \ldots, w_{(m+1)n}\} = \{\al_i^{s} \mid i \in [n], \, s \in [0\td m]\}$,
        \item for all $i \in [n]$ and $s \in [m]$, the letter $\al_i^{s-1}$ appears before $\al_i^{s}$,
        \item for all $i, j \in [n]$ and $s, t \in [m]$, if $\al_i^{s-1}$ appears before $\al_j^{t-1}$ then $\al_i^{s}$ appears before $\al_j^{t}$.
    \end{itemize}
\end{definition}

An $(m,n)$-sketch $w$ indicates the order on the set $\{x_i + s\mid i \in [n],~ s\in [0;m]\}$ for a point $(x_1, \ldots, x_n)$ in a region of the $m$-Catalan arrangement, with $\al_i^s$ representing the number $x_i + s$.

We define an ordering on the letters of an $(m,n)$-sketch as follows:
\begin{definition}
    Let $w$ be an $(m,n)$-sketch. We define an ordering $\leq_w$ on the letters of $w$, that is, on the set $\{\al_i^{s} \mid i \in [n], s\in [0\td m]\}$ by $\al_i^{s} <_w \al_j^{t}$ if and only if $\al_i^{s}$ is before $\al_j^{t}$ in $w$. 
    
    We define $\beta(w)$ to be the region of the $m$-Catalan arrangement made of the points $x \in \R^n$ such that  $x_i - x_j < s$ if and only if $\al_i^0 <_w \al_j^s$.
\end{definition}

\begin{lemma}[{~\cite[Proposition 8.1]{Bernardi}}]\label{sketchregionbij} The map $\beta$ is a bijection between $(m,n)$-sketches and regions of the $m$-Catalan arrangement.  
\end{lemma}

The $\bS$-braid arrangement is a subarrangement of the $m$-Catalan arrangement for $m = \max\bigcup_{i,j\in[n]} S_{i,j}^+$. Hence, any region of the $\bS$-braid arrangement is a union of regions of the $m$-Catalan arrangement. We further define $\beta_{\bS}(w)$ to be the region of $\As$ containing the $m$-Catalan region $\beta(w)$. 

\begin{definition}
We define $\trips = \{(i,j,s) \mid i,j \in [n], (s \in S_{i,j}^+) \text{ or } (s = 0 \text{ and } i < j)\}$. 

The $\bS$-braid arrangement $\As$ is called \emph{transitive} if for all $i,j\in[n]$ and $s,t\in\mathbb{N}$,
    $$(i,j,s)\notin \trips \text{ and } (j,k,t) \notin \trips \implies (i,k,s+t)\notin \trips.$$
\end{definition}

In~\cite{Bernardi}, Bernardi further defined the notion of \emph{local maximality} on the $(m,n)$-sketches. 

\begin{definition}
    Let $\bS = (S_{i,j})_{1 \leq i < j \leq n}$ be a collection of finite sets of integers and let $m = \max\bigcup_{i,j\in[n]} S_{i,j}^+$. A $(m,n)$-sketch is called \emph{$\bS$-locally maximal} if for all $j \in [n]$, if $s$ is maximum such that $\al_j^{s}$ is immediately followed by $\al_i^{0}$ for some $i\in [n]$, then $(i,j,s) \in \trips$. We denote the set of $\bS$-locally maximal sketches by $\Ls$.
\end{definition}

The following result was proved in~\cite[Sec. 8]{Bernardi}.
\begin{theorem}[{~\cite{Bernardi}}]\label{AtLeastOneL}
For any tuple $\bS=(S_{i,j})_{i<j}$,    the map $\beta_{\bS}: \Ls \ra R(\As)$ is surjective. Moreover, if $\As$ is transitive, then $\beta_{\bS}$ is bijective. 
\end{theorem}

\section{Right inverse to the generalized Pak-Stanley labeling}\label{RightInverseSec}
In this section, we first define a map from $\bS$-locally maximal sketches to $\Ds$-parking functions that is consistent with the \GPS labeling of a region. We then define a map from $\Ds$-parking functions to $\bS$-locally maximal sketches and show that when composed with $\beta_{\bS}$, this is a right inverse to the \GPS labeling. 

Let $\bS = (S_{i,j})_{1\leq i < j \leq n}$ be a tuple of finite sets of integers, and let $m = \max\bigcup_{i,j\in[n]} S_{i,j}^+$. By definition, the \GPS labeling $\lambda(R)$ of a region of $R$ of $\mA_{\bS}$ is $p = (p_1, \ldots, p_n)$, where 
\begin{align*}
    p_i = |\{(j,s) \mid j\in [n], ~ s\in S_{i,j}^+  \textrm{ such that }s < x_i - x_j ~\}|.
\end{align*}

For a $(m,n)$-sketch $w$, we let $\Phi(w)$ be the $\Ds$-parking function $p = (p_1, \ldots, p_n)$, where 
\begin{align*}
    p_i = |\{(j,s) \mid \, j\in [n], \, s\in S_{i,j}^+\textrm{ such that }\al_{j}^{s} <_w \al_i^{0} ~\}|.
\end{align*}
It is clear from the definition of $\be_\bS$ that $\Phi = \lambda \circ \beta_{\bS}$.

We will now define a right inverse of $\Phi$. Precisely,  we will define a map $\Psi$ from the set of $\Ds$-parking functions to the set of $\bS$-local maximum $(m,n)$-sketches. 

Let $p = (p_1, \ldots, p_n)$ be a $\Ds$-parking function. We define the word $w=w_1\ldots w_\ell=\Psi(p)$ by determining its letters $w_r$ one at a time according to the following process which is illustrated in Figure \ref{fig:exp-Phi}:

We maintain a $n$-tuple $P_r\in\ZZ^n$ which we initialize at $P_1=(p_1,\ldots,p_n)$, and an ordered list $O_r$ of indices in $[n]$ which we initialize to be empty: $O_1=\emptyset$.
We determine the letter $w_r$ of $\Psi(p)$, starting from $r=1$, according to the following rule: 

\begin{adjustwidth}{2em}{0pt}
    \noindent \textbf{Case 1}: The tuple $P_{r}$ has a zero entry. 

    Then, the letter $w_r$ is $\al_k^{0}$, where $k$ is the rightmost zero entry of $P_r$.
     We obtain $P_{r+1}$ from $P_{r}$ by subtracting $1$ from coordinate $k$ and every coordinate $i$ with a positive entry such that $0 \in S_{i,k}^+$. We obtain $O_{r+1}$ from $O_{r}$ by appending $k$ at the end of the list $O_r$.

    \noindent \textbf{Case 2}: The tuple $P_{r}$ does not have a zero entry, but $O_r$ is non-empty. 

    Then, the letter $w_r$ is $\al_k^{s}$, where $k$ is the first element of $O_r$ and $-s$ is the $k^{th}$ entry of~$P_r$. We obtain $P_{r+1}$ from $P_{r}$ by subtracting $1$ from coordinate $k$ and every coordinate $i$ with a positive entry such that $s \in S_{i,k}^+$. We obtain the list $O_{r+1}$ from $O_{r}$ by removing the first element $k$ of the list, and if $s<m$ appending $k$ at the end of the list $O_r$.
\end{adjustwidth}

This process terminates when neither condition is satisfied, and returns the word $w_1\ldots w_\ell$ constructed thus far. See Figure \ref{fig:exp-Phi} for some examples.

\begin{figure}[h]
\begin{subfigure}{\textwidth}
\centering
\includegraphics[clip,width=0.9\textwidth]{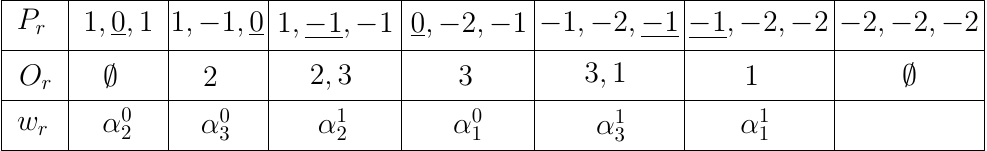}
\caption{$\Psi(1,0,1)$ for the $1$-Shi arrangement in dimension $3$.}
\end{subfigure}
\begin{subfigure}{\textwidth}
\centering
\includegraphics[clip,width=0.9\textwidth]{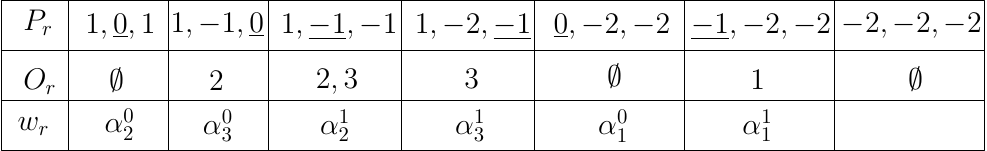}
\caption{$\Psi(1,0,1)$ for the $1$-Shi arrangement minus the hyperplane $\{x_1 - x_2 = 1\}$ (see Figure~\ref{fig:labelings}).}
\end{subfigure}
\caption{Computing $\Psi(p)$ for the parking function $p=(1,0,1)$ for two different $\bS$.}\label{fig:exp-Phi}
\end{figure}

\begin{lemma}\label{PsiWD}
For any $\Ds$-parking function $p$, the word $\Psi(p)$ is an $\bS$-locally maximal sketch.
\end{lemma}
\begin{proof} Let $p=(p_1,\ldots,p_n) \in \parkS$ and let $w=\Psi(p)$.

We first show that $w$ is an $(m,n)$-sketch. It is easy to see that when the procedure defining $\Psi$ terminates, the entries $(q_1,\ldots,q_n)$ of $P_r$ are all either positive or equal to $-m$ (for the indices that ever entered the list $O_r$). Let $U\subset[n]$ be the set of indices of the positive entries of $P_r$, and let $V=[n]\setminus U$.  Note that for $i \in U$, we have subtracted $1$ from position $i$ for every element of $S_{i,j}^+$ for every $j\in V$. 
Hence, $p_i = q_i + \sli_{j \in V} |S_{i,j}^+| \geq 1+ \sli_{j \in V} |S_{i,j}^+|.$ 
Moreover, in the digraph $D_\bS$, for $i\in U$, one has
$|\{(i,j) \mid j \notin U\}| = 1+ \sli_{j \in V} |S_{i,j}^+|.$
This gives $p_i \geq |\{(i,j) \mid j \notin U\}|$ for all $i \in U$. This contradicts $p\in \parkS$ unless $U=\emptyset$. We conclude that $U=\emptyset$ and $V=[n]$. This implies that all the letters $\al_i^s$ for $i\in[n]$ and $s\in[0;m]$ appear once in $w=\Psi(p)$. It is then clear from the definition that $w$ is an $(m,n)$-sketch (since the list $O_r$ is ``first in first out'')
    
It remains to show that $w$ is $\bS$-locally maximal.
Suppose the letter $w_r=\al_j^{s}$ is immediately followed by  $\al_i^{0}$ in $w$. We will show $(i,j,s) \in \trips$.

    If $s=0$ and $i<j$, then clearly $(i,j,s) \in \trips$. If $s = 0$ and $i > j$, as $w_r=\al_j^{0}$, the rightmost zero in $P_{r}$ has to be at coordinate $j$. In other words, for $k > j$, the $k^{th}$ coordinate of $P_{r}$ is non-zero. But, for the next letter to be $\al_i^{0}$, the $i^{th}$ coordinate of $P_{r+1}$ must be zero. As we obtain $P_{r+1}$ by subtracting $1$ from every coordinate $k$ such that $0 \in S_{k,j}^+$, we must have $0 \in S_{i,j}^+$, so $(i,j,s) \in \trips$. 
    
    If $s \neq 0$, then $P_{r}$ had no zeros. Now, if the next letter of $\Psi(p)$ is $\al_i^{0}$, we must have that $P_{r+1}$ has a zero at position $i$ and that this is in fact the rightmost zero in $P_{r+1}$. By definition of $\Psi$, $P_{r+1}$ is obtained from $P_{r}$ by subtracting $1$ from every coordinate $k$ such that $s \in S_{k,j}^+$. Hence $s \in S_{i,j}^+$, and $(i,j,s) \in \trips$. Thus, $w$ is $\bS$-locally maximal. 
\end{proof}

The main result of this section is the following:
\begin{theorem}\label{IdParking}
For every collection of sets $\bS=(S_{i,j})_{i<j}$, one has $\Phi \circ \Psi = \emph{Id}_{\parkS}$. 
\end{theorem}

We first need some definitions and a lemma.

\begin{definition}
    We define $\Ms$ to be the set of $(m,n)$-sketches such that if $\al_j^{s}$ is immediately followed by $\al_i^{0}$, then $(i,j,s) \in \trips$. 
    
    We define $\Ns=\text{Im}(\Psi)$, that is, the image by $\Psi$ of $\parkS$.
\end{definition}

\begin{remark}\label{rk:inclusion}
    It is clear that $\Ms \sse \Ls$. Moreover we have shown $\Ns \sse \Ms$ in the proof of Lemma~\ref{PsiWD}. 
\end{remark}

\begin{lemma}\label{keylemma}
    Let $p = (p_1, \ldots, p_n)\in \parkS$, and let $\al_k^{s}$ be the $r^{th}$ letter of $w = \Psi(p)$. Let $P_{r} = (q_1, \ldots, q_n)$. For all $i\in[n]$ such that $q_i \geq 0$, one has 
    $$
        p_i - q_i = |\{\al_j^{t} <_w \al_k^{s} \mid t \in S_{i,j}^+\}|.
    $$
\end{lemma}
\begin{proof}
    By definition of $\Psi$, the tuple $P_{r}$ is determined by $p$ and the first $r-1$ letters $w_1,\ldots,w_{r-1}$ of $\Psi(p)$.  For $\ell \in [r-1]$, we consider how the letter $w_\ell$ affects the $i^{th}$ coordinate of $P_{r}$.

    Let  $w_{\ell} = \al_j^{t}$. By definition of $\Psi$, for such a letter, $1$ is subtracted from every coordinate~$k$ with a positive entry such that $t \in S_{k,j}^+$ that is, $1$ is subtracted from $p_i$ for every letter of the form  $\al_j^{t}$ with $t \in S_{i,j}^+$ in the first $r-1$ letters of $\Psi(p)$.  
    
    Hence, $p_i - q_i = |\{\al_j^{t} <_w \al_k^{s} \mid t \in S_{i,j}^+\}|.$
\end{proof}

We now prove Theorem~\ref{IdParking}. 
\begin{proof}[Proof of Theorem~\ref{IdParking}]
    Let $p = (p_1, \ldots, p_n) \in \parkS$. Let $w = \Psi(p)$ and $\Phi(w) = p' = (p_1', \ldots, p_n')$. By definition of $\Phi$, 
    \begin{align*}
        p_i' = |\{(j,s) \mid \al_{j}^{s} <_w \al_i^{0}, ~ j\in [n], ~ s\in S_{i,j}^+\}| = |\{\al_{j}^{s} <_w \al_i^{0} \mid s\in S_{i,j}^+\}| .
    \end{align*}   
    We wish to show that $p = p'$. Suppose $\al_i^{0}$ is the $r^{th}$ letter of $w$. Then, for $P_{r} = (q_1, \ldots, q_n)$, we have $q_i = 0$, and further, by Lemma~\ref{keylemma}, 
    \begin{align*}
        p_i = p_i - q_i = |\{\al_j^{t} <_w \al_i^{0} \mid t \in S_{i,j}^+\}| = p_i'.
    \end{align*}
    Hence $p = p'$, that is, $\Phi\circ \Psi = \text{Id}_{\parkS}$. 
\end{proof}

\begin{remark}
    As $\Phi = \lambda \circ \beta_{\bS}$, Theorem~\ref{IdParking} implies that $\lambda \circ \beta_{\bS} \circ \Psi = \text{Id}_{\parkS}.$  Hence the map $\beta_{\bS} \circ \Psi$ is a right inverse to the \GPS labeling.
    
Note that this implies Mazin's surjectivity result \cite{MR3721647} for $\bS$-braid arrangements, and hence for all deformations of the braid arrangement by Remark~\ref{rk:S-is-general}.
\end{remark}


\section{Bijectivity of the generalized Pak-Stanley labeling for $\me$-arrangements}\label{suffconds}
In this section, we establish certain sufficient conditions for the bijectivity of the \GPS labeling, and use them to prove the following theorem.

\begin{theorem}\label{mebij}
    The \GPS labeling $\lambda$ is bijective for $\me$-arrangements with inverse $\beta_{\bS} \circ \Psi$. 
\end{theorem}

The proof of Theorem \ref{mebij} starts with the following lemma.

\begin{lemma}\label{NLM}
    Let $\As$ be an $\bS$-braid arrangement. 
    If $\Ns = \Ms = \Ls$, then the \GPS labeling $\lambda$ is bijective. 
    Conversely, if  $\lambda$ and $\be_\bS:\Ls\to R(\As)$ are both bijective then $\Ns = \Ms = \Ls$.    
\end{lemma}

\newcommand{\bLS}{\be_\bS^L}
\newcommand{\bNS}{\be_\bS^N}

\begin{proof} Recall that $\text{Im}(\Psi)=\Ns \sse \Ms \sse \Ls$. Let $\bNS$ and $\bLS$ be the restriction of the map $\be_\bS$ to $\Ns$ and $\Ls$ respectively.

By Theorem \ref{IdParking},  $\lambda\circ \bNS\circ \Psi=\text{Id}_{\parkS}$. Hence $\bNS$ is injective, and  $\lambda$ is bijective if and only if $\bNS: \Ns\to R(\As)$ is surjective. By  Theorem~\ref{AtLeastOneL}, $\bLS:\Ls\to R(\As)$ is surjective.  Hence if $\Ns =\Ls$ then $\bNS$ is surjective and $\lambda$ is bijective. Conversely, if $\bLS$ is bijective, then  $\bNS$ is surjective if and only if $\Ns =\Ls$.
\end{proof}

We now want to prove that  $\Ns = \Ms = \Ls$ for every $\me$-arrangement $\As$. To this aim, we will establish some sufficient conditions for the identities $\Ns = \Ms$ and $\Ms = \Ls$ to hold.

\begin{definition}
    We say an $\bS$-braid arrangement $\As$ has \emph{Property Y}, or $(Y)$ holds, if
    \begin{equation}\tag{Y}\label{Y}
        \forall i,j,k \in [n], \, \forall s \geq 0 \text{ with } s>0 \text{ or } j<i, \text{ if } s \notin S_{i,j}^+, \text{ then } s+t \notin S_{k,j}^+ \text{ for all } t>0, k\neq i.
    \end{equation}
\end{definition}

\begin{lemma}\label{Yneccsuff}
    Let $\As$ be an $\bS$-braid arrangement. Then~\eqref{Y} holds if and only if $\Ms = \Ls$. 
\end{lemma}
\begin{proof}
Suppose~\eqref{Y} holds. Assume for contradiction that there exists $w \in \Ls\setminus \Ms$. Then, there is $(i,k,s) \notin \trips$ such that $\al_k^s$ is followed immediately by $\al_i^{0}$ in $w$. Now, as $w \in \Ls$, we will have $\al_k^{s+t}$ followed immediately by $\al_j^{0}$ for some $t>0$, with $(j,k,s+t) \in \trips$. As $(i,k,s) \notin \trips$, we have $s \notin S_{i,k}^+$ with $s > 0$ or $k<i$. By~\eqref{Y}, $s+t \notin S_{j,k}^+$. Hence, $s+t = 0$ with $j<k$, which contradicts $t > 0$. Hence $\Ls = \Ms$.

 Suppose now $\Ms = \Ls$. Suppose~\eqref{Y} does not hold. Then two cases may arise.

\begin{adjustwidth}{2em}{0pt}
    \noindent \textbf{Case 1}: There exists $s > 0$ such that $s \notin S_{i,k}^+$ and $s+t \in S_{j,k}^+$ for some $i,j,k \in [n]$, $t>0$. 

    Then, consider the $(m,n)$-sketch
    $$w = \al_n^0 \ldots \al_k^0\ldots \al_1^0\ldots \al_k^{s}\al_i^0\ldots \al_k^{s+t}\al_j^0\ldots$$
    where the first $(n-2)$ letters are $\al_p^0$ with $p \neq i, j$ in descending order of $p$. Then, $w $ is in $\Ls\setminus \Ms$, which is a contradiction. 

    \noindent \textbf{Case 2}: There exists $i > k$ such that $0 \notin S_{i,k}^+$ and $t \in S_{j,k}^+$ for some $j \in [n]$, $t>0$. 

    Then, consider  the $(m,n)$-sketch
    $$w = \al_n^0\ldots \al_k^0\al_i^0 \ldots \al_1^0 \ldots \al_k^t\al_j^0\ldots$$
    where the letters before $\al_k^0$ are $\al_p^0$, $p > k$, $p \neq j$ in descending order of $p$ and the letters after $\al_i^0$ are $\al_p^0$, $p<k<i$, $p \neq j$ in descending order of $p$. Again, $w $ is in $\Ls\setminus \Ms$, which is a contradiction. 
\end{adjustwidth}

\end{proof}    

\begin{definition}
    We say a $\bS$-braid arrangement $\As$ has \emph{Property X}, or $(X)$ holds, if \begin{equation}\tag{X}\label{X}
        \forall 1\leq i< j \leq n, k \in [n], \, \forall s \in S_{j,k}^+ \text{ either } s \in S_{i,k}^+ \text{ or } (s = 0 \text{ and } i<k).
    \end{equation}
\end{definition}

\begin{lemma}\label{Xsuff}
    Let $\As$ be an $\bS$-braid arrangement. If~\eqref{X} holds, then $\Ns = \Ms$. 
\end{lemma}
\begin{proof}
We suppose that~\eqref{X} holds. 
We know $\Ns\subseteq\Ms$ and want to prove $\Ns = \Ms$. By Theorem \ref{IdParking}, the map $\Phi:\Ns\to\parkS$ is surjective, thus it suffices to show that $\Phi:\Ms\to\parkS$ is injective.

Suppose for contradiction that there are distinct sketches $w=w_1\ldots w_{(m+1)n}$ and $w'=w_1'\ldots w_{(m+1)n}'$ both in $\Ms$ such that $\Phi(w)=\Phi(w')$. Let $a\in [(m+1)n]$ be the smallest index such that $w_a\neq w_a'$.\smallskip

\begin{adjustwidth}{2em}{0pt}
    \noindent \textbf{Claim:} We cannot have $w_a = \al_i^{s}$ and $w_a' = \al_j^{t}$ with $s\neq 0$ and $t \neq 0$. 

    Suppose $w_a = \al_i^{s}$ and $w_a' = \al_j^{t}$ with $s\neq 0$, $t \neq 0$. Since $s \neq 0$, and $w$ is an $(m,n)$-sketch, $\al_i^{s-1}$ will be one of the first $a-1$ letters of $w$. Similarly, since $t \neq 0$, $\al_j^{t-1}$ will be one of the first $a-1$ letters of $w'$. As $w_k = w_k'$ for all $k \in [a-1]$, $\al_i^{s-1}$ and $\al_j^{t-1}$ are two of the first $a-1$ letters of both $w$ and $w'$. 

    Now in $w$, we have $\al_i^{s}$ appears before $\al_j^{t}$, and hence $\al_i^{s-1}$ must appear before $\al_j^{t-1}$ as $w$ is an $(m,n)$-sketch. But, in $w'$, $\al_i^{s}$ appears after $\al_j^{t}$ and hence $\al_i^{s-1}$ must appear after $\al_j^{t-1}$ as $w'$ is an $(m,n)$-sketch. This contradicts the fact that $w_k = w_k'$ for all $k \in [a-1]$. Hence, our claim holds.\smallskip
\end{adjustwidth}

From this claim, we can assume without loss of generality (up to exchanging the role of $w$ and $w'$) that $w_a=\al_i^0$ and $w_a'=\al_j^s$ with either $s>0$ or ($s=0$ and $j<i$). 

By looking at the $i^{th}$ coordinate of the parking function $\Phi(w)=\Phi(w')$ we get
\begin{equation}\label{eq:pi-equal}
|\{\al_k^{t} <_{w} \al_i^{0} \mid k\neq i, ~t \in S_{i,k}^+\}|=|\{\al_k^{t} <_{w'} \al_i^{0} \mid k \neq i,~ t \in S_{i,k}^+\}|.
\end{equation}
Let $b>a$ be such that $w_{b}'=\al_i^0$, and let us denote $w_{b-\ell}'=\al_{k_\ell}^{t_\ell}$ for all $\ell\in[0;b-a]$. By~\eqref{eq:pi-equal} we get $t_\ell\notin S_{i,k_\ell}^+$ for all $\ell\in [b-a]$. 

Let us now prove that $t_\ell=0$ and $i\leq k_\ell$ for all $\ell\in[0;b-a]$, by induction on $\ell$. The base case $\ell=0$ holds since $w_b'=\al_i^0$. Suppose, by induction, that  $t_{\ell-1}=0$ and $i\leq k_{\ell-1}$ for some $\ell\in [b-a]$. Since $w'\in\Ms$ we have $(k_{\ell-1},k_{\ell},t_{\ell})\in\trips$. Thus either $t_\ell\in S_{k_{\ell-1},k_{\ell}}^+$ or ($t_\ell=0$ and $k_{\ell-1}<k_{\ell}$). In the second case we get $t_\ell=0$ and $i\leq k_\ell$, as wanted. In the case $t_\ell\in S_{k_{\ell-1},k_{\ell}}^+$, Property \eqref{X} implies that either $t_\ell\in S_{i,k_\ell}^+$   or ($t_\ell=0$ and $i\leq k_\ell$). Since we have established above that $t_\ell\notin S_{i,k_\ell}^+$, we get $t_\ell=0$ and $i\leq k_\ell$, as wanted.

For $\ell=b-a$, this give $w_a'=\al_j^s$ with $s=t_{b-a}=0$ and $j=k_{b-a}\geq i$, which gives a contradiction.   
\end{proof}

Lemmas~\ref{NLM},~\ref{Yneccsuff} and~\ref{Xsuff} immediately imply the following.
\begin{cor}\label{XYsuff}
If a $\bS$-braid arrangement satisfies~\eqref{X} and~\eqref{Y}, then the \GPS labeling $\lambda$ is bijective. 
\end{cor}

\begin{figure}[h]
        \centering
        \includegraphics[width=\linewidth]{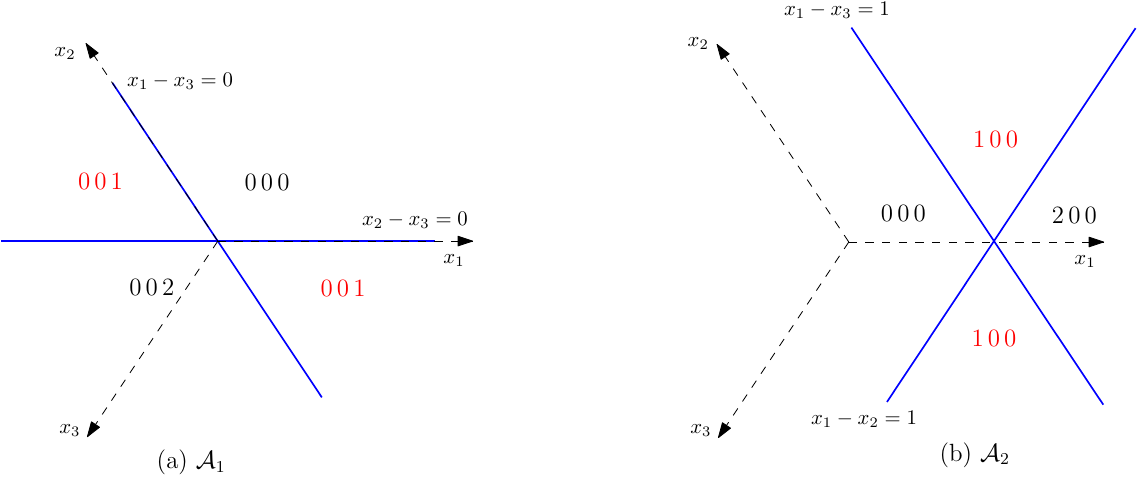}
        \caption{Two arrangements for which either $(X)$ or $(Y)$ does not hold and the \GPS labeling is not injective.}
        \label{fig:nonexample}
    \end{figure}
    
\begin{remark}
    For the arrangement $\mA_1 = \{x_1 - x_3 = 0, x_2 - x_3 = 0\}$, one can verify that~\eqref{Y} holds. However, as $0 \in S_{3,1}^+$ and $0 \notin S_{2,1}^+$,~\eqref{X} does not hold. As seen in Fig.~\ref{fig:nonexample}(a), the GPS labeling of $\mA_1$ is not injective.

    Similarly, for the arrangement $\mA_2 = \{x_1 - x_2 = 1, x_1 - x_3 = 1$\}, one can verify that~\eqref{X} holds. However, as $0 \notin S_{3,2}^+$ and $1 \in S_{1,2}^+$,~\eqref{Y} does not hold. As seen in Fig.~\ref{fig:nonexample}(b), the GPS labeling of $\mA_2$ is not injective.
\end{remark}

Finally, it is easy to check that $\me$-arrangements  satisfy~\eqref{X} and~\eqref{Y} which completes the proof of Theorem \ref{mebij}.


\section{Bijectivity of the generalized Pak-Stanley labeling for transitive arrangements}\label{TransitiveArrSec}

In this section, we focus on \emph{transitive} $\bS$-braid arrangements and prove the following result.
\begin{theorem}~\label{mebij-only}
    Let $\As$ be a transitive $\bS$-braid arrangement. Then the following are equivalent
    \begin{itemize}[noitemsep]
    \item[(a)] the \GPS labeling $\lambda$ is bijective 
    \item[(b)] $\Ns = \Ls = \Ms$
    \item[(c)] \eqref{X} and~\eqref{Y} hold 
    \item[(d)] $\As$ is a $\me$-arrangement
    \end{itemize}
\end{theorem}
For a graphical arrangement $\mA$ (a $\bS$-braid arrangement for which $S_{i,j} \in \{\emptyset, \{0\}\}$ for all $i,j \in [n]$), \eqref{Y} always holds, while \eqref{X} can be rewritten as
\begin{equation}\label{eq:Xgraphical}
\textrm{there are no indices } i<j<k\textrm{ in }[n]\textrm{ such that } H_{j,k,0}\notin \mA\textrm{ and } H_{i,k,0}\in \mA.
\end{equation}
Hence, Theorem \ref{mebij-only} implies the following.
\begin{cor}
The \GPS labeling $\lambda$ is bijective for a transitive graphical arrangement $\mA$ 
if and only if \eqref{eq:Xgraphical} holds.
\end{cor}
This corollary could also be obtained from a result of Mazin and Miller~\cite{MMGraphical}, which states that $\lambda$ is bijective for a (not necessarily transitive) graphical arrangement if and only if 
\begin{equation}\label{eq:graphical-bij}
\textrm{there are no indices } i<j<k\textrm{ in }[n]\textrm{ such that } H_{i,j,0}\in\mA,~H_{j,k,0}\notin \mA\textrm{ and } H_{i,k,0}\in \mA.
\end{equation}
Indeed, for graphical arrangements, the transitivity property can be rewritten as:
\begin{equation}\label{eq:transitive-graphical}
\textrm{there are no indices } i<j<k\textrm{ in }[n]\textrm{ such that } H_{i,j,0}\notin\mA,~H_{j,k,0}\notin \mA\textrm{ and } H_{i,k,0}\in \mA,
\end{equation}
and it is clear that the arrangements satisfying \eqref{eq:graphical-bij} and \eqref{eq:transitive-graphical} are those satisfying \eqref{eq:Xgraphical}.

The rest of this section is dedicated to the proof of Theorem~\ref{mebij-only}. 
The equivalence between (a) and (b) is a direct consequence of  Theorem~\ref{AtLeastOneL} and Lemma~\ref{NLM}. Next we prove the  equivalence between (b) and (c).

\begin{lemma}\label{Xnecc}
    Let $\As$ be a transitive $\bS$-braid arrangement such that~\eqref{Y} holds and $\Ns = \Ms$. Then~\eqref{X} holds. 
\end{lemma}
\newcommand{\un}[1]{\underline{#1}}
\begin{proof}
    Suppose for contradiction that~\eqref{X} does not hold. 

    Suppose first that there exist $k < i < j$ such that $0 \in S_{j,k}^+\setminus S_{i,k}^+$. We take the minimal such $i$ given $k$ and $j$. The transitivity of $\As$ implies $0 \in S_{j,i}^+$. 
    Now, consider the sketches
    $$w = \al_n^0 \ldots \al_{i+1}^0 \un{\al_k^0} \al_j^0 \un{\al_i^0} \al_{i-1}^0\ldots \al_1^0\ldots~ \text{ and } w' = \al_n^0 \ldots \al_{i+1}^0 \un{\al_i^0} \al_j^0 \un{\al_k^0}\al_{i-1}^0 \ldots \al_1^0\ldots,$$
    where the letters from $\al_n^0$ to $\al_1^0$ (except $\al_i^0, \al_j^0,$ and $\al_k^0$) are in decreasing order of subscript.
    We have underlined the letters which have moved between $w$ and $w'$ (the omitted suffix in $w$ and $w'$ corresponds to repeating the indicated prefix but with superscripts $1,2,\ldots,m$, so as to get $(m,n)$-sketches).
    It is easy to see that $w, w' \in \Ms$ (using the minimality assumption on $i$),    and $p_{\ell} = p_{\ell}'$ for all $\ell \in [n]$. So $\Phi(w) = \Phi(w')$. Since $\Phi$ is injective on $\Ns=\Ms$, we reach a contradiction. 
    
    So there must exist $s> 0$ such that $s \in S_{j,k}^+\setminus S_{i,k}^+$ for some $i<j$ and $k$. Note that $0 \in S_{j,i}^+$ as if not, by transitivity, as $s \notin S_{i,k}^+$ we have $s \notin S_{j,k}^+$, which would be a contradiction.
    Let $i^* = \min\{\ell \in [n] \mid 0 \in S_{i, \ell}^+ \}\cup\{i+1\}$ and $j^* = \min\{\ell \in [i] \mid 0 \in S_{j,\ell}^+\}$. By \eqref{Y}, we have $s-1 \in S_{\ell,k}^+$ for all $\ell \neq j$, otherwise we would have $s \notin S_{j,k}^+$. If $i^*>j^*$, we  consider the sketches
    $$w = \al_k^0 \ldots \al_k^{s-1}\al_{n^*}^0\ldots \al_1^0 \al_k^s \un{\al_j^0} \un{\al_i^0}\ldots ~\text{ and } w' = \al_k^0\ldots \al_k^{s-1}\al_{n^*}^0\ldots\al_{i^*}^0\un{\al_i^{0}} \ldots \al_{j^*}^0\un{\al_j^0}\ldots \al_1^0 \al_k^s\ldots,$$ 
    where $n^* = \max{[n]\setminus\{i,j,k\}}$, 
    and $\al_{n^*}^0 \ldots \al_1^0$ represent all the letters $\al_{\ell}^0$ with $\ell \neq i,j,k$ in decreasing order of $\ell$. We have again underlined the letters which move between $w$ and $w'$ and omitted the suffixes (which are uniquely determined by the indicated prefixes for $(m,n)$-sketches). If $i^*\leq j^*$, then we instead consider the sketch 
    $$w'= \al_k^0\ldots \al_k^{s-1}\al_{n^*}^0\ldots\al_{i^*}^0\un{\al_i^{0}} \un{\al_j^0}\ldots \al_1^0 \al_k^s\ldots.$$ 
    It is easy to see that (in either cases) the sketches $w, w'$ are in $\Ms$. 
    Let $\Phi(w) = (p_1,\ldots, p_n)$ and $\Phi(w') = (p_1', \ldots, p_n')$. Clearly, $p_{\ell} = p_{\ell}'$ for all $\ell \neq i,j$. Further, by choice of $i^*$, as $s \notin S_{i,k}^+$ and $j>i$, $p_i = p_i'$. Finally, by choice of $j^*$, as $s \in S_{j,k}^+$ and $0 \in S_{j,i}^+$, $p_j = p_j'$. Hence $\Phi(w) = \Phi(w')$ for $w \neq w'$. Since $\Phi$ is injective on $\Ns=\Ms$, we again reach a contradiction.
   Hence~\eqref{X} holds. 
\end{proof}
The equivalence of (b) and (c) follows directly from Lemmas \ref{Yneccsuff}, \ref{Xnecc} and~\ref{Xsuff}. It only remains to prove the equivalence of (c) and~(d).

First, it is easy to check~\eqref{X} and~\eqref{Y} hold for $\me$-arrangements hence (d) implies~(c). 
We now suppose that~\eqref{X} and~\eqref{Y} hold for an arrangement $\As$. We want to show that $\As$ is a $\me$-arrangement. Fix $k \in [n]$. Let $s$ be the least non-negative integer such that $s \notin S_{i,k}^+$ with $s > 0$ or $i>k$.  Then, as~\eqref{Y} holds, for all $j\neq i,k$, one has $s+t \notin S_{j,k}^+$ for all $t> 0$. Hence, $[1, s - 1] \sse S_{j,k}^+ \sse [0, s]$ for all $j\neq i,k$.

    Next, suppose $s+t \in S_{i,k}^+$ for some $t>0$. Then, $s \in S_{j,k}^+$ for all $j \neq i$, as otherwise, as~\eqref{Y} holds, $s+t \notin S_{i,k}^+$. Further, as~\eqref{X} holds, $s+t \in S_{j,k}^+$ for all $j < i$, a contradiction unless $i = 1$. But $s \in S_{j,k}^+$ and $s \notin S_{1,k}^+$ contradicts~\eqref{X}. Hence, $[1, s-1] \sse S_{i,k}^+ \sse [0, s-1]$. 

    Finally, suppose $s \in S_{j,k}^+$ for some $j \neq i$. Then we define $m_k = s$ and for all $\ell\neq k$ we define  $\eps_{\ell,k}=1$ if $s \notin S_{\ell, k}$ (and either $s>0$ or $k<\ell$) and  $\eps_{\ell,k}=0$ otherwise. As~\eqref{X} holds, for $p > \ell$, if $s \in S_{p,k}^+$, $s \in S_{\ell, k}^+$, hence we must have $\eps_{p,k} = 1$ for all $p > \ell$. If $s \notin S_{j,k}^+$ for all $j \in [n]$, we define $m_{k} = s-1$ and $\eps_{\ell,k}=0$ for all $\ell\neq k$. Hence $\As$ is an $\me$-arrangement. 
    This completes the proof of Theorem \ref{mebij-only}.

    \begin{remark} 
        We have already established that for a non-negative integer $m$, the $m$-Shi and $m$-Catalan arrangements are examples of $\me$-arrangements. 
        Further, one can interpolate between the $(m-1)$-Catalan arrangement and the $m$-Catalan arrangement using $\me$-arrangements: there exists a sequence of  $\me$-arrangements $\mA_0,\mA_1,\ldots,\mA_k$, where $\mA_0$ is the $(m-1)$-Catalan arrangement, $\mA_k$ is the $m$-Catalan arrangement, and $\mA_i$ is obtained from $\mA_{i-1}$ by adding a single hyperplane for all $i\in[k]$. This sequence can also be made to include the $m$-Shi arrangement.
    \end{remark}
    
\begin{figure}[h]
    \centering
    \includegraphics[width=0.96\linewidth]{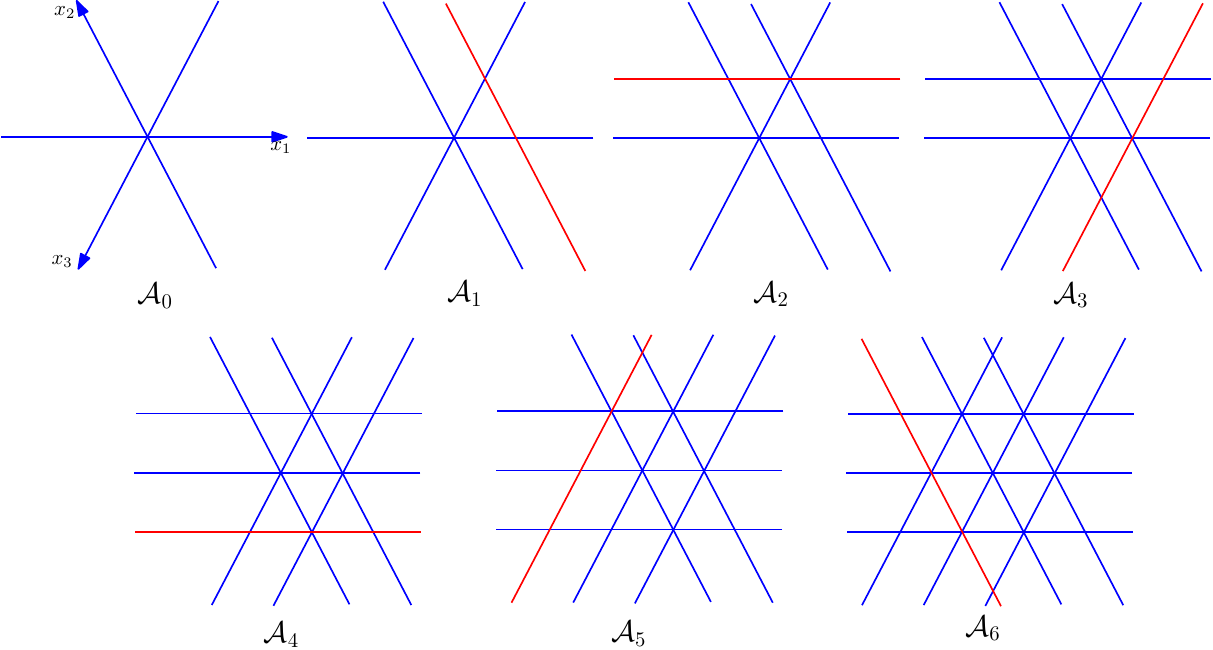}
    \caption{An interpolation from the braid arrangement to the $1$-Catalan arrangement. The hyperplane added at each step is shown in red.}
    \label{fig:interpolation}
\end{figure}

\begin{example}
    Figure~\ref{fig:interpolation} demonstrates an interpolation between the braid arrangement and the $1$-Catalan arrangement using $\me$-arrangements. The hyperplanes added, in order, are $H_{1,3,1}$, $H_{2,3,1}$, $H_{1,2,1}$, $H_{3,2,1}$, $H_{2,1,1}$ and $H_{3,1,1}$. In this example, $\mA_3$ is the $1$-Shi arrangement.
\end{example}

\noindent \textbf{Acknowledgment:}  We thank the FPSAC referees for their useful suggestions.

\printbibliography

\end{document}